\documentclass[12pt,reqno]{amsart}

\usepackage{color}
\usepackage{amsmath, amsthm, amssymb}
\usepackage{amsfonts}

\usepackage[ansinew]{inputenc}
\usepackage[dvips]{epsfig}
\usepackage{graphicx}
\usepackage[english]{babel}
\usepackage{hyperref}

\newtheorem{thm}{Theorem}[section]

\newtheorem{prop}[thm]{Proposition}
\newtheorem{remark}[thm]{Remark}

\newtheorem{lem}[thm]{Lemma}
\newtheorem{lemma}[thm]{Lemma}
\newtheorem{cor}[thm]{Corollary}

\theoremstyle{definition}
\newtheorem{defi}[thm]{Definition}
\def\RR{{\mathbb R}}
\def\R{{\mathbb R}}

\begin{document}

\title[Entire solutions to semilinear nonlocal equations in $\RR^2$]{Entire solutions to semilinear \\ nonlocal equations in $\RR^2$}

\author[X. Ros-Oton]{Xavier Ros-Oton}
\address{The University of Texas at Austin, Department of Mathematics, 2515 Speedway, Austin, TX 78751, USA}
\email{ros.oton@math.utexas.edu}

\author[Y. Sire]{Yannick Sire}
\address{Universit\'e Aix-Marseille, i2m, Marseille, France}
\email{yannick.sire@univ-amu.fr}


\maketitle

\begin{abstract}
We consider entire solutions to $L u= f(u)$ in $\RR^2$, where $L$ is a general nonlocal operator with kernel $K(y)$. Under certain natural assumtions on the operator $L$, we show that any stable solution is a 1D solution.
In particular, our result applies to any solution $u$ which is monotone in one direction. Compared to other proofs of the De Giorgi type results on nonlocal equations, our method is the first successfull attempt to use the Liouville theorem approach to get flatness of the level sets.
\end{abstract}

\tableofcontents

\section{Introduction}

This paper is concerned with the study of bounded solutions to semilinear equations
\begin{equation}\label{pb}
L u = f(u)\,\,\,\mbox{in}\,\,\RR^2
\end{equation}
for nonlocal elliptic operators of the form
\begin{equation}\label{L}
L u(x)=\textrm{PV}\int_{\RR^2} \bigl( u(x)-u(x+y) \bigr) K(y)\,dy.
\end{equation}
More precisely, we study the 1D symmetry of stable solutions to \eqref{pb}.

When $L$ is the Laplacian $-\Delta$, the interest in this type of problems goes back to a conjecture of de Giorgi \cite{DeG}, and several works have been devoted to study
the Allen-Cahn equation
\[-\Delta u=u-u^3\,\,\,\mbox{in}\,\,\RR^n.\]
Solutions to this problem are by now quite well understood. Indeed, in dimension $2$, Ghoussoub and Hui proved that any monotone solution is 1D (see \cite{GG}); in dimension $3$ it has been proved by Ambrosio and Cabr\'e in \cite{AC}. For dimensions between $4$ and $8$, under a natural assumption on the limit profiles, Savin proved the conjecture \cite{savin}.

On the other hand, in the last years several works have been devoted to the study of semilinear nonlocal equations of the type
\[(-\Delta)^s u = f(u)\,\,\,\mbox{in}\,\,\RR^n.\]
Here, $(-\Delta)^s$ is the fractional Laplacian, which corresponds to $K(y)=c|y|^{-n-2s}$ in~\eqref{L}, $s \in (0,1)$.
In particular, in the paper \cite{sire}, one of the authors and Valdinoci proved that any bounded stable solution of the previous equation in $n=2$ is one-dimensional.
Another proof of this result, established independently at the same time, can be found also in \cite{cabS2}. This is the idea of this latter proof that we will follow. 
In the papers \cite{CabC,cabre}, Cabr\'e and Cinti extended this result to $n=3$ for $s\geq\frac12$.

These results are all based on the extension problem for $(-\Delta)^s$, which transforms the nonlocal problem in $\R^n$ into a local one in $\R^{n+1}_+$ (see \cite{cafSil}). 
However, no result was known for any other nonlocal operator of the form \eqref{L}.

The novelty of our approach is that it does not use the so-called Caffarelli-Silvestre extension. 
We have been aware, while we were writing the paper, that Cinti, Serra and Valdinoci have another proof of our result using a quantitative stability argument (see \cite{CSV}).

The goal of the present paper is to establish this type of symmetry result in two dimensions for a class of nonlocal operators of the form \eqref{L}.

\subsection*{Assumptions on the kernel}

We make the following assumptions on $L$:
\begin{itemize}
\item[(H1)] The operator $L$ is of the form \eqref{L}, with the kernel $K$ satisfying $K \geq 0$ and $K(y)=K(-y)$.
Moreover, $K$ has \emph{compact support} in $B_1$
\[K\equiv0\quad\textrm{in}\ \R^n\setminus B_1,\]
and
\[\int_{B_1}|y|^2 K(y)dy \leq C.\]

\item[(H2)] The operator $L$ satisfies the following \emph{Harnack inequality}: If $\varphi>0$ is a solution to $L\varphi+c(x)\varphi=0$ in $\R^2$, with $c(x)\in L^\infty(\R^n)$, then
    \[\sup_{B_1(x_0)}\varphi \leq C \inf_{B_1(x_0)}\varphi\]
    for any $x_0\in \R^2$.

\item[(H3)] The following \emph{H\"older estimate} holds: If $w$ is a bounded solution to $Lw=g$ in $B_1$, with $g\in L^\infty(B_1)$, then
\[\|w\|_{C^\alpha(B_{1/2})}\leq C\bigl(\|g\|_{L^\infty(B_1)}+\|w\|_{L^\infty(\R^n))}\bigr)\]
for some constants $\alpha>0$ and $C$.
Moreover, the space $H_K(\R^2)$, defined as the closure of $C^\infty_c(\R^2)$ under the norm
\[\|w\|_{H_K(\R^2)}^2=\frac12\int_{\R^2}\int_{\R^2}\bigl(w(x)-w(x+y) \bigr)^2 K(y)\,dy\,dx,\]
is compactly embedded in $L^2(\R^2)$.
\end{itemize}

As we will see, assumptions (H1) and (H2) are important for our purposes.
The assumption (H3) is mainly to make sure that solutions $u$ satisfy $|\nabla u|\in L^\infty(\R^2)$, and to prove the equivalence between the two definitions of stability below.

As said before, we consider stable solutions.
As in the classical case of local equations, we have two equivalent definitions of stability: a variational one, and a non-variational one.
We will show in Lemma \ref{linear} that, thanks to (H2)-(H3), these two definitions are equivalent.

\begin{defi}\label{defi-stable}
Assume that (H2)-(H3) hold.

A bounded solution $u$ to \eqref{pb} is said to be \emph{stable} if
$$
\frac12\int_{\R^2}\int_{\R^2}\bigl( \psi(x)-\psi(x+y) \bigr)^2 K(y)\,dy\,dx \geq \int_{\R^2} f'(u)\psi^2
$$
for every $\psi \in C^\infty_c(\R^2)$.

Equivalently, $u$ is said to be stable if there exists a bounded solution $\varphi>0$ to
$$
L \varphi =f'(u)\varphi\quad\textrm{in}\ \R^2
$$
i.e. $\varphi$ solves the linearization of the problem \eqref{pb}.
\end{defi}

Our main result is the following.

\begin{thm}\label{thm1}
Let $L$ be given by \eqref{L}, and assume that (H1)-(H2)-(H3) hold.
Let $f$ be a locally Lipschitz nonlinearity, and $u$ be any bounded stable solution to \eqref{pb}.

Then, $u$ is a 1D function, i.e., $u(x)=w(x\cdot a)$ for some $a\in \mathbb S^1$.
\end{thm}

Estimates of the form (H3) have been widely studied, and are known for many different classes of kernels; see the results of Silvestre \cite{silvestreIndiana} and also Kassmann-Mimica \cite{KM}.
Still, in order to keep our results cleaner, we prefer to state our result under the only assumptions (H1)-(H2)-(H3).

Similarly, Harnack inequalities have been widely studied and are known for different classes of kernels $K(y)$; see for instance a rather general form of the Harnack inequality in \cite{CKP}.
Notice that in our case, we need an Harnack inequality with a zero order term in the equation.
It has been proved when the integral operator is the pure fractional Laplacian in \cite{CabS1} and refined in \cite{TX}.
It is by now well known that the Harnack inequality may fail depending on the kernel $K$ under consideration, and a characterization of the classes of kernels for which it holds is out of the scope of this paper.
Thus, in order to keep the statements of our results clean, we have decided to state them under the general assumptions (H1)-(H2)-(H3).

\vspace{3mm}

An important example of stable solutions are monotone solutions, i.e., solutions $u$ for which $\partial_e u>0$ for some direction $e$.
Indeed, one just has to take $\varphi=\partial_e u$ in Definition \ref{defi-stable}.

As a direct consequence of Theorem \ref{thm1}, we find the following.

\begin{cor}\label{cor1}
Let $L$ be given by \eqref{L}, and assume that (H1)-(H2)-(H3) hold.
Let $f$ be a locally Lipschitz nonlinearity, and $u$ be any bounded solution to \eqref{pb}.

If $\partial_e u>0$ for some direction $e$, then $u$ is a 1D function.
\end{cor}

In the next Section we prove our main result, Theorem \ref{thm1}, and Corollary \ref{cor1}.
After this, we prove in Section \ref{sec3} the equivalence between the two definitions of stability.

\section{Proof of the main result}
\label{sec2}

Theorem \ref{thm1} will follow from the following two results.
The first one gives an equation for
\begin{equation}\label{sigma}
\sigma=\frac{\partial_{x_i}u}{\varphi},
\end{equation}
where $u$ is the solution of Theorem \ref{thm1} and $\varphi>0$ is given by Definition \ref{defi-stable}.

\begin{prop}\label{Step1}
Let $L$, $u$, and $\varphi$ be as in Theorem \ref{thm1}, and let $\sigma$ be defined by~\eqref{sigma}.
Then,
\begin{equation}\label{strange-equation}
\int_{\R^2}\int_{\R^2}\eta^2(x)\sigma(x)\bigl(\sigma(x)-\sigma(z)\bigr)\varphi(x)\varphi(z)K(x-z)dx\,dx=0
\end{equation}
for all $\eta\in C^\infty_c(\R^2)$.
\end{prop}

The second one is the following.

\begin{prop}\label{Step2}
Let $L$, $u$, and $\varphi$ be as in Theorem \ref{thm1}, and let $\sigma$ be defined by~\eqref{sigma}.
Assume \eqref{strange-equation} holds for all $\eta\in C^\infty_c(\R^2)$.
Then, $\sigma$ is constant.
\end{prop}

We first prove Proposition \ref{Step1}.

\begin{proof}[Proof of Proposition \ref{Step1}]
First, notice that
\[L(\sigma \varphi)=\sigma L \varphi +\varphi L \sigma -I(\varphi, \sigma)\]
where
\[I(\varphi, \sigma)(x):=\int_{\R^2} \bigl(\sigma(x)-\sigma(z)\bigr)\bigl(\varphi(x)-\varphi(z)\bigr) K(x-z)\,dz.\]
On the other hand, since $\sigma\varphi=\partial_{x_i}u$, then
\[L(\sigma \varphi)=L\partial_{x_i}u=f'(u)\partial_{x_i} u\]
Moreover, since $\varphi$ solves $L\varphi=f'(u)\varphi$, then
\[\sigma L \varphi=f'(u)\partial_{x_i}u.\]

Hence, we end up with the equation
\[\begin{split}
0&=\varphi\,L\sigma-I(\varphi,\sigma)\\
&= \varphi(x)\int_{\R^2} \bigl(\sigma(x)-\sigma(z)\bigr)K(x-z)\,dz\\
&\qquad\qquad -\int_{\R^2} \bigl(\sigma(x)-\sigma(z)\bigr)\bigl(\varphi(x)-\varphi(z)\bigr) K(x-z)\,dz\\
&= \int_{\R^2}\bigl(\sigma(x)-\sigma(z)\bigr)\varphi(x)\varphi(z)K(x-z)dz.
\end{split}\]
Multiplying by $\eta^2(x)\sigma(x)$ and integrating in $x$, we find that $\sigma$ satisfies \eqref{strange-equation}.
\end{proof}

To prove Proposition \ref{Step2}, we will need the following.

\begin{lem}\label{key-bound}
Let $L$, $u$, and $\varphi$ be as in Theorem \ref{thm1}, and let $\sigma=\partial_{x_i}u/\varphi$.

Let $R>1$, and let $\eta\in C^\infty(\R^2)$ be such that
\[\eta\equiv1\quad\textrm{in}\ B_R,\qquad \eta\equiv0\quad\textrm{in}\ \R^2\setminus B_{2R},\qquad |\nabla \eta|\leq \frac{C}{R}.\]
Then,
\[\int_{\R^2}\int_{\R^2}\bigl(\sigma(x)+\sigma(z)\bigr)^2\bigl(\eta(x)-\eta(z)\bigr)^2\varphi(x)\varphi(z)K(x-z)dx\,dz \leq C\]
for some constant $C$ independent of $R$.
\end{lem}

\begin{proof}
Recall that, thanks to (H1), we have that $K\equiv0$ in $\R^2\setminus B_1$.
Moreover,
\[\bigl(\eta(x)-\eta(x+y)\bigr)^2\leq \frac{C|y|^2}{R^2}\,\chi_{B_{3R}}(x)\qquad \textrm{for}\ |y|\leq 1.\]
Therefore, it suffices to bound
\[\frac{1}{R^2}\int_{B_{3R}}dx\int_{B_1}dy\,\bigl(\sigma(x)+\sigma(x+y)\bigr)^2\varphi(x)\varphi(x+y)\,|y|^2K(y).\]
But notice that, since $|\nabla u|\leq C$, then
\[|\sigma|=\frac{|\nabla u|}{\varphi}\leq \frac{C}{\varphi},\]
and thus
\[\bigl(\sigma(x)+\sigma(x+y)\bigr)^2\leq \frac{C}{\varphi^2(x)}+ \frac{C}{\varphi^2(x+y)}.\]
This yields
\[\bigl(\sigma(x)+\sigma(x+y)\bigr)^2\varphi(x)\varphi(x+y)\leq C\left(\frac{\varphi(x)}{\varphi(x+y)}+\frac{\varphi(x+y)}{\varphi(x)}\right).\]

Now, by the Harnack inequality (H2), we have
\[\frac{\varphi(x)}{\varphi(x+y)}+\frac{\varphi(x+y)}{\varphi(x)}\leq C\]
for all $x\in \R^2$ and $y\in B_1$.
Hence, we have
\[\begin{split}
&\frac{1}{R^2}\int_{B_{3R}}dx\int_{B_1}dy\,\bigl(\sigma(x)+\sigma(x+y)\bigr)^2\varphi(x)\varphi(x+y)\,|y|^2K(y) \\
&\qquad \qquad \leq \frac{C}{R^2}\int_{B_{3R}}dx\int_{B_1}dy\,\,|y|^2K(y)\leq \frac{C}{R^2}\,|B_{3R}|\leq C,\end{split}\]
and the Lemma is proved.
\end{proof}

We can now give the:

\begin{proof}[Proof of Proposition \ref{Step2}]
First, symmetrizing \eqref{strange-equation} in $x$ and $z$, we get
\[\int_{\R^2}\int_{\R^2}\bigl(\eta^2(x)\sigma(x)-\eta^2(z)\sigma(z)\bigr)\bigl(\sigma(x)-\sigma(z)\bigr)\varphi(x)\varphi(z)K(x-z)dz\,dx=0\]
for all $\eta\in C^\infty_c(\R^2)$.
Thus, using that
\[2\bigl(\eta^2(x)\sigma(x)-\eta^2(z)\sigma(z)\bigr)=\bigl(\sigma(x)-\sigma(z)\bigr)\bigl(\eta^2(x)+\eta^2(z)\bigr)+
\bigl(\sigma(x)+\sigma(z)\bigr)\bigl(\eta^2(x)-\eta^2(z)\bigr),\]
we find the new equation
\begin{equation}\label{aust}\begin{split}
I(R)&:=\int_{\R^2}\int_{\R^2}\bigl(\sigma(x)-\sigma(z)\bigr)^2\bigl(\eta^2(x)+\eta^2(z)\bigr)\varphi(x)\varphi(z)K(x-z)dz\,dx  \\
&= - \int_{\R^2}\int_{\R^2}\bigl(\sigma^2(x)-\sigma^2(z)\bigr)\bigl(\eta^2(x)-\eta^2(z)\bigr)\varphi(x)\varphi(z)K(x-z)dz\,dx.
\end{split}\end{equation}

Let us now take $\eta$ such that
\[\eta\equiv1\quad\textrm{in}\ B_R,\qquad \eta\equiv0\quad\textrm{in}\ \R^2\setminus B_R,\]
with $R>1$ is large enough.
Then, using the Cauchy-Schwarz inequality and that $\bigl(\eta(x)+\eta(z)\bigr)^2\leq 2\bigl(\eta^2(x)+\eta^2(z)\bigr)$,
\begin{equation}\label{Ambrosio-Cabre}\begin{split}
|I(R)|^2&\leq \left(\int\int_{\{\eta(x)\neq\eta(z)\}}\bigl(\sigma(x)-\sigma(z)\bigr)^2\bigl(\eta^2(x)+\eta^2(z)\bigr)\varphi(x)\varphi(z)K(x-z)dx\,dz \right)\\
&\qquad\qquad \times \left(\int_{\R^2}\int_{\R^2}\bigl(\sigma(x)+\sigma(z)\bigr)^2\bigl(\eta(x)-\eta(z)\bigr)^2\varphi(x)\varphi(z)K(x-z)dx\,dz \right).
\end{split}\end{equation}
Moreover, by Lemma \ref{key-bound}, the last integral is bounded (uniformly in $R$).

This means that $I(R)\leq C$, and letting $R\rightarrow\infty$,
\[\int_{\R^2}\int_{\R^2}\bigl(\sigma(x)-\sigma(z)\bigr)^2\varphi(x)\varphi(z)K(x-z)dz\,dx\leq C.\]
But then, letting $R\rightarrow \infty$ in \eqref{Ambrosio-Cabre}, we have
\[\int\int_{\{\eta(x)\neq\eta(z)\}}\bigl(\sigma(x)-\sigma(z)\bigr)^2\bigl(\eta^2(x)+\eta^2(z)\bigr)\varphi(x)\varphi(z)K(x-z)dx\,dz\longrightarrow 0,\]
and therefore,
\[\int_{\R^2}\int_{\R^2}\bigl(\sigma(x)-\sigma(z)\bigr)^2\varphi(x)\varphi(z)K(x-z)dz\,dx=0.\]
Since $\varphi>0$, this means that $\sigma$ is constant, and the Proposition is proved.
\end{proof}

Finally, we give the:

\begin{proof}[Proofs of Theorem \ref{thm1} and Corollary \ref{cor1}]
The result follows from Propositions \ref{Step1} and \ref{Step2}.
Indeed, using these results we find that any partial derivative $\partial_{x_i}u$ satisfies
\[\sigma=\frac{\partial_{x_i}u}{\varphi}=c_i\]
for some constants $c_1$ and $c_2$.

This means that $\partial_e u\equiv0$ in $\R^2$ for $e=c_2e_1-c_1e_2$, and thus $u$ is a 1D solution.
\end{proof}

\begin{remark}
It is important to remark that that for monotone solutions (say in the $x_2$ direction), assumption (H3) is actually not needed. 
One just needs to take $\varphi=\partial_{x_2} u$ in the previous argument, in which we did not used (H3). 
However, this assumption is required for the argument in the next section.
\end{remark}

\section{A characterization of stability}
\label{sec3}

This section is devoted to the following lemma.

\begin{lemma}\label{linear}
Assume that (H2) and (H3) hold, and let $u$ be any bounded solution to~\eqref{pb}.
Then, the following statements are equivalent.
\begin{itemize}
\item[(i)] The following inequality
\[\frac12\int_{\R^2}\int_{\R^2}\bigl( \xi(x)-\xi(x+y) \bigr)^2 K(y)\,dz\,dx \geq \int_{\R^2} f'(u)\xi^2\]
holds for every $\xi \in C^\infty_c(\R^2)$.

\item[(ii)] There exists a bounded solution $\varphi>0$ to
\[L \varphi=f'(u)\varphi\quad\textrm{in}\ \R^2\]
i.e. $\varphi$ solves the linearization of the problem \eqref{pb}.
\end{itemize}
\end{lemma}

\begin{proof}
Let us first show that (ii) $\Longrightarrow$ (i).

Let $\xi\in C^\infty_c(\R^2)$.
Using $\xi^2/\varphi$ as a test function in the equation $L\varphi=f'(u)\varphi$, we find
\[\int_{\R^2}f'(u)\xi^2=\int_{\R^2}\frac{\xi^2}{\varphi}\,L\varphi.\]
Next, we use the integration by parts type formula
\[\int_{\R^2}v\,Lw=\frac12\int_{\R^2} B(v,w),\]
where
\[B(v,w)(x):=\int_{\R^2}\int_{\R^2}\bigl(v(x)-v(y)\bigr)\bigl(w(x)-w(y)\bigr)K(x-y)dy.\]
We find
\[\int_{\R^2}f'(u)\xi^2=\frac12\int_{\R^2}B\left(\varphi,\,\xi^2/\varphi\right).\]
Now, it is immediate to check that
\[\frac{\xi^2(x)}{\varphi(x)}-\frac{\xi^2(y)}{\varphi(y)}=\left(\xi^2(x)-\xi^2(y)\right)\frac{\varphi(x)+\varphi(y)}{2\varphi(x)\varphi(y)}
-\left(\varphi(x)-\varphi(y)\right)\frac{\xi^2(x)+\xi^2(y)}{2\varphi(x)\varphi(y)},\]
and this yields
\[\begin{split}
2\int_{\R^2}f'(u)\xi^2=&
\int_{\R^2}\int_{\R^2}\bigl(\varphi(x)-\varphi(y)\bigr)\left(\xi^2(x)-\xi^2(y)\right)\frac{\varphi(x)+\varphi(y)}{2\varphi(x)\varphi(y)}\,K(x-y)dx\,dy\\
&-\int_{\R^2}\int_{\R^2}\bigl(\varphi(x)-\varphi(y)\bigr)^2\ \frac{\xi^2(x)+\xi^2(y)}{2\varphi(x)\varphi(y)}K(x-y)dx\,dy.
\end{split}\]
Let us now show that
\begin{equation}\label{volem}
\begin{split}
&\Theta(x,y):=\bigl(\varphi(x)-\varphi(y)\bigr)\left(\xi^2(x)-\xi^2(y)\right)\frac{\varphi(x)+\varphi(y)}{2\varphi(x)\varphi(y)}\\
&\qquad\qquad\qquad\qquad\qquad-\bigl(\varphi(x)-\varphi(y)\bigr)^2\ \frac{\xi^2(x)+\xi^2(y)}{2\varphi(x)\varphi(y)}\leq \bigl(\xi(x)-\xi(y)\bigr)^2.
\end{split}\end{equation}
Once this is proved, then we will have
\[2\int_{\R^2} f'(u)\xi^2\leq \int_{\R^2}\int_{\R^2}\bigl( \xi(x)-\xi(y) \bigr)^2 K(x-y)dx\,dy,\]
and thus the result will be proved.

To establish \eqref{volem}, it is convenient to write $\Theta$ as
\[\begin{split}
&\Theta(x,y)=2\bigl(\varphi(x)-\varphi(y)\bigr)\bigl(\xi(x)-\xi(y)\bigr)\frac{\xi(x)+\xi(y)}{\varphi(x)+\varphi(y)}\cdot
\frac{\bigl(\varphi(x)+\varphi(y)\bigr)^2}{4\varphi(x)\varphi(y)}\\
&\qquad\qquad-\bigl(\varphi(x)-\varphi(y)\bigr)^2\cdot\left(\frac{\xi(x)+\xi(y)}{\varphi(x)+\varphi(y)}\right)^2 \frac{2\xi^2(x)+2\xi^2(y)}{\bigl(\xi(x)+\xi(y)\bigr)^2}\cdot\frac{\bigl(\varphi(x)+\varphi(y)\bigr)^2}{4\varphi(x)\varphi(y)}.
\end{split}\]
Now, using the inequality
\[\begin{split}
&2\bigl(\varphi(x)-\varphi(y)\bigr)\bigl(\xi(x)-\xi(y)\bigr)\frac{\xi(x)+\xi(y)}{\varphi(x)+\varphi(y)}\leq\\
&\qquad\qquad\qquad\qquad\qquad \leq \bigl(\xi(x)-\xi(y)\bigr)^2+\bigl(\varphi(x)-\varphi(y)\bigr)^2\cdot\left(\frac{\xi(x)+\xi(y)}{\varphi(x)+\varphi(y)}\right)^2,
\end{split}\]
we find
\[\begin{split}
&\qquad\Theta(x,y)\leq \bigl(\xi(x)-\xi(y)\bigr)^2\frac{\bigl(\varphi(x)+\varphi(y)\bigr)^2}{4\varphi(x)\varphi(y)}\,+\\
&+\bigl(\varphi(x)-\varphi(y)\bigr)^2\cdot\left(\frac{\xi(x)+\xi(y)}{\varphi(x)+\varphi(y)}\right)^2\cdot
\frac{\bigl(\varphi(x)+\varphi(y)\bigr)^2}{4\varphi(x)\varphi(y)}\cdot\left\{1-\frac{2\xi^2(x)+2\xi^2(y)}{\bigl(\xi(x)+\xi(y)\bigr)^2}\right\}.
\end{split}\]
But since
\[1-\frac{2\xi^2(x)+2\xi^2(y)}{\bigl(\xi(x)+\xi(y)\bigr)^2}=-\,\frac{\bigl(\xi(x)-\xi(y)\bigr)^2}{\bigl(\xi(x)+\xi(y)\bigr)^2},\]
then
\[\begin{split}
\Theta(x,y)&\leq
\bigl(\xi(x)-\xi(y)\bigr)^2\frac{\bigl(\varphi(x)+\varphi(y)\bigr)^2}{4\varphi(x)\varphi(y)} - \bigl(\varphi(x)-\varphi(y)\bigr)^2\cdot \frac{\bigl(\xi(x)-\xi(y)\bigr)^2}{4\varphi(x)\varphi(y)}\\
&=\frac{\bigl(\xi(x)-\xi(y)\bigr)^2}{4\varphi(x)\varphi(y)}\left\{\bigl(\varphi(x)+\varphi(y)\bigr)^2-\bigl(\varphi(x)-\varphi(y)\bigr)^2\right\}\\
&= \bigl(\xi(x)-\xi(y)\bigr)^2.\end{split}\]
Hence \eqref{volem} is proved, and the result follows.

\vspace{4mm}

Let us now show that (i) $\Longrightarrow$ (ii). Assume (i) holds.

Let $\xi$ be a smooth compactly supported function in $\mathbb R^2$ and consider the quadratic form
$$
\mathcal Q_R(\xi)=\frac12\int_{\mathbb R^n} \int_{\mathbb R^2}\bigl( \xi(x)-\xi(x+y)\bigr)^2 K(y)\,dx\,dy-\int_{B_R}f'(u)\xi^2\,dx.
$$
Let us define $H_K(\R^2)$ as the closure of $C^\infty_c(\R^2)$ under the norm
\[\|w\|_{H_K(\R^n)}:=\int_{\mathbb R^n} \int_{\mathbb R^n}\bigl(w(x)-w(x+y)\bigr)^2 K(y)\,dx\,dy.\]

Let $\lambda_R$ be the infimum of $\mathcal Q_R$ among the class $\mathcal S_R$ defined by
$$
\mathcal S_R= \left \{ \xi \in H_K(\mathbb R^2)\,:\,\,\,\xi\equiv 0\,\,\mbox{in}\,\,\mathbb R^n \backslash B_R,\,\,\int_{B_R} \xi^2 =1 \right \}
$$
Since the functional $\mathcal Q_R$ is bounded from below (since $f'(u)$ is bounded) and thanks to the compactness assumption in (H3), the infimum $\lambda_R$ is attained for a function $\phi_R \in \mathcal S_R$.
Moreover, because of the stability condition (i), we have $\lambda_R\geq0$.

It is easy to see that $\phi_R \geq 0$ ---since if $\phi$ is minimizer then $|\phi|$ is also a minimizer.
Thus, the function $\phi_R\geq0$ is a solution, not identically zero, of the problem
\[\left \{
\begin{array}{c}
L \phi_R=f'(u)\phi_R+\lambda_R \phi_R , \,\,\,\mbox{in}\,\,\,B_R,\\
\phi_R=0\,\,\,\mbox{in}\,\,\,\mathbb R^n \backslash B_R.
\end{array} \right.\]
It follows from the strong maximum principle that $\phi_R>0$ in $B_R$.

On the other hand, for any $R<R'$ we have
\[\int_{B_{R'}}\phi_{R}L\phi_{R'}=\int_{B_{R'}}\phi_{R'}L\phi_{R}<\int_{B_{R}} \phi_{R'}L\phi_{R}.\]
This inequality follows from the fact that $\phi_{R}=0$ in $B_{R'}\setminus B_{R}$, and thus $L\phi_{R}<0$ in that annulus.
Hence, using the equations for $\phi_{R}$ and $\phi_{R'}$ we deduce that
\[\lambda_{R'}\int_{B_{R}}\phi_{R}\phi_{R'}<\lambda_{R}\int_{B_{R}}\phi_{R}\phi_{R'}.\]
Therefore, $\lambda_{R'}<\lambda_{R}$ for any $R'>R$.
In particular, $\lambda_R>0$ for all $R$.

Now consider the problem
\[\left \{
\begin{array}{c}
L \varphi_R=f'(u)\varphi_R , \,\,\,\mbox{in}\,\,\,B_R,\\
\varphi_R=c_R\,\,\,\mbox{in}\,\,\,\mathbb R^2 \backslash B_R.,
\end{array} \right.\]
for $c_R>0$.
The solution to this problem can be found by writing $\psi_R=\varphi_R-c_R$, which solves
\[\left \{
\begin{array}{c}
L \psi_R=f'(u)\psi_R+c_Rf'(u) , \,\,\,\mbox{in}\,\,\,B_R,\\
\psi_R=0\,\,\,\mbox{in}\,\,\,\mathbb R^n \backslash B_R.
\end{array} \right.\]
It is immediate to check that the energy functional associated to this problem is bounded from below and coercive, thanks to the inequality $\lambda_R>0$.

Next we claim that $\varphi_R>0$ in $B_R$.
To show this, we use $\varphi_R^-$ as a test function for the equation for $\varphi_R$.
We find
\[\begin{split}
\frac12\int_{\mathbb R^n} \int_{\mathbb R^n}\bigl(\varphi_R(x)-\varphi_R(x+y)\bigr)\bigl(\varphi_R^-(x)-\varphi_R^-(x+y)\bigr) K(y)\,dx\,dy&=\int_{B_R}f'(u)\varphi_R\varphi_R^-\\
&=-\int_{B_R}f'(u)|\varphi_R^-|^2.
\end{split}\]
Now, since
\[\bigl(\varphi_R(x)-\varphi_R(x+y)\bigr)\bigl(\varphi_R^-(x)-\varphi_R^-(x+y)\bigr)\leq -\bigl(\varphi_R^-(x)-\varphi_R^-(x+y)\bigr)^2,\]
this yields
\[\mathcal Q_R(\varphi_R^-)=\frac12\int_{\mathbb R^n} \int_{\mathbb R^n}\bigl(\varphi_R^-(x)-\varphi_R^-(x+y)\bigr)^2 K(y)\,dx\,dy-\int_{B_R}f'(u)|\varphi_R^-|^2\,dx\leq 0.\]
Since $\lambda_R>0$, this means that $\varphi_R^-\equiv0$, and thus $\varphi_R\geq0$.
By the strong maximum principle, $\varphi_R>0$ in $B_R$.

Finally, let us choose the constant $c_R>0$ so that $\varphi_R(0)=1$.
Then, by Harnack inequality (H2) and the H\"older regularity (H3), the function $\psi_R$ converges to a function $\varphi>0$ in $\RR^2$ and satisfying (ii).
\end{proof}

\section{Satisfying assumption (H2)}

In this section, we comment on the validity of the Harnack inequality with a zero order. Consider equation in $(H2)$, i.e.
$$
L\varphi= f'(u)\varphi
$$
where $\varphi >0$ in $\mathbb R^2$ and $L$ has compact support.
Following the approach of \cite{CKP}, the crucial part to prove the Harnack inequality we need is a logarithmic lemma.
They proved:

\begin{lemma}[\cite{CKP2}]\label{log}
Let $L$ be an operator of the form
\[Lu(x)={\rm PV}\int_{\R^n}\bigl(u(x)-u(x+y)\bigr)K(x,y)dy,\]
with $K$ satisfying
\[\frac{\lambda}{|y|^{n+2s}}\,\chi_{B_1}(y)\leq K(x,y)\leq \frac{\Lambda}{|y|^{n+2s}},\qquad s\in(0,1).\]
Let $u \in H^s{(\mathbb R^n})$ be a weak solution of
\[Lu =0\,\,\mbox{in}\,\,\Omega\]
with suitable boundary conditions in $\R^n\setminus\Omega$.
Assume that $u \geq 0$ in $\mathbb R^2$.
Then for any $r>0$ and any $d>0$
$$
\int_{B_r} \int_{B_r} K(x,y)\left| \log \left  ( \frac{d+u(x)}{d+u(y)}\right)\right|^2 \leq Cr^{n-2s}
$$
for some constant $C$.
\end{lemma}

The previous lemma is based on a test function argument multiplying the equation by $\phi^2 /u$ where $\phi$ is a standard cut-off. In our case since the zero order term $c(x)$ is bounded, the same proof holds paying the price of an additional term in the right hand side.
However, since our kernel is compactly supported then this term is uniformly controlled, and hence we get for this class of kernels the desired assumption (H2).
\begin{remark}
Using the so-called Caffarelli-Silvestre extension \cite{cafSil}, such an Harnack inequality has been proved in \cite{CabS1} and \cite{TX}. 
\end{remark}

\section*{Acknowledgements}

Y.S. is supported by the ANR grants "HAB" and "NONLOCAL" and the ERC grant "EPSILON".

X.R. was supported by grants MTM2011-27739-C04-01 (Spain), and 2009SGR345 (Catalunya)

\bibliographystyle{alpha}
\bibliography{biblio}

\end{document}